\documentclass[12pt]{amsart}
\usepackage{amsmath,amssymb,amsthm,mathtools}
\usepackage{mathrsfs,bm}
\usepackage{thmtools,thm-restate}
\usepackage{tikz-cd}
\usepackage{tikz}
\usepackage{enumitem}
\usepackage{graphicx}
\usetikzlibrary{matrix,arrows.meta,bending}
\usepackage{color}
\usepackage[numbers]{natbib}
\usepackage{booktabs}
\usepackage{longtable}
\usepackage{pgfplots}
\pgfplotsset{compat=1.14} 
\usetikzlibrary{arrows.meta,decorations.pathreplacing,positioning,calc} 

\usepackage[a4paper, left=3cm, right=3cm, top=3cm, bottom=3cm, footskip=15mm]{geometry}

\newtheorem{theorem}{Theorem}[section]

\newtheorem{proposition}[theorem]{Proposition}
\newtheorem{corollary}[theorem]{Corollary}
\newtheorem{conjecture}[theorem]{Conjecture}
\theoremstyle{definition}
\newtheorem{definition}[theorem]{Definition}

\theoremstyle{remark}
\newtheorem{remark}[theorem]{Remark}

\numberwithin{equation}{section}

\newcommand{\inn}{~ \hat{\in}~ }

\makeindex

\usepackage{fancyhdr}
\usepackage{calc} 

\AtBeginDocument{%
  \setlength{\textwidth}{\dimexpr\paperwidth - 6cm\relax}%
  \setlength{\oddsidemargin}{\dimexpr 3cm - 1in\relax}%
  \setlength{\evensidemargin}{\dimexpr 3cm - 1in\relax}%
  \setlength{\marginparwidth}{2.5cm}%
}

\begin{document}

\title{The theory of the Collatz process and the method of dynamical balls}

\author{T. Agama}
\address{Department of Mathematics, African Institute for Mathematical science, Ghana
}
\email{theophilus@aims.edu.gh/emperordagama@yahoo.com}



\date{\today}



\begin{abstract}
In this paper, we introduce and develop the theory of the Collatz process and the method of dynamical balls. We leverage this theory to study the Collatz conjecture. This theory also has a subtle connection with the infamous problem of the distribution of Sophie Germain primes. We provide several formulations of the Collatz conjecture in this language. Furthermore, we introduce and develop the notion of dynamical systems induced by fixed $a\in \mathbb{N}$ and their associated induced dynamical balls. We develop tools to study problems that require determining the convergence of certain sequences generated by iterating on a fixed integer.
\end{abstract}

\maketitle

\section{Introduction and motivation}

The $3x+1$ (Collatz) problem is a paradigmatic instance of an elementary iterative arithmetic dynamical system whose global behaviour resists all known methods despite its extremely simple definition. Writing
$$
f(n)=\begin{cases}n/2,& n\equiv 0\pmod 2,\\3n+1,& n\equiv 1\pmod 2,\end{cases}
$$
the Collatz conjecture asserts that for every $n\in\mathbb N$ some iterate $f^{s}(n)$ attains the cycle $\{1,2\}$; equivalently, every forward orbit reaches $1$.  This plain statement has produced a broad literature of heuristic, computational, and partial theoretical results and surveys; see, for instance, the authoritative treatments and bibliographies in \cite{lagarias2010ultimate,lagarias19853,chamberland2003update,guy2004unsolved} and the expository remarks on related iterative sequences in \cite{pickover1991computers}. Despite decades of work, these investigations have yielded convincing heuristics and deep structural observations but no definitive global resolution.\\

Two overarching difficulties explain the elusive nature of the problem. First, the map $f:\mathbb N\to\mathbb N$ mixes multiplicative and additive behaviours (division by $2$ versus affine expansion $3n+1$) so that statistical models-while suggestive-do not capture delicate arithmetic obstructions.  Second, the forward dynamics alone conceals a rich inverse (preimage) tree structure whose arithmetic geometry appears central: control of backward orbits seems necessary to convert global heuristics into rigorous convergence statements. A fruitful recent perspective (pursued in several works and surveys) is therefore to treat the $3x+1$ map as an arithmetic dynamical system and to couple forward statistical estimates with a careful study of inverse branches and their combinatorial organization.  Our work develops this philosophy by introducing two complementary tools: the \emph{Collatz process} (a refined forward/backward bookkeeping of orbits) and the \emph{method of dynamical balls} (a metric-combinatorial language to quantify how orbit radii evolve).\\

The first contribution of this paper is a systematic formalization of the Collatz process. Unlike the standard forward-only viewpoint, the Collatz process records both forward iterates and the \emph{infimum} of backward preimages, and thereby isolates the arithmetic and parity phenomena that govern orbit growth and contraction. This bookkeeping clarifies notions such as \emph{generator}, \emph{order}, and \emph{index} of an orbit and yields immediate structural consequences (uniqueness of finite generators, basic parity constraints, and relations between index and order for residues and prime inputs).  These reformulations provide compact, algebraically transparent restatements of several folklore observations and place them in a language suitable for further quantitative analysis.\\

The second and principal contribution is the introduction of \emph{dynamical systems induced by sequences} together with their associated \emph{dynamical balls}.  Given any map $f:\mathbb{N}\to\mathbb{N}$ and a base point $a$, the $k^{th}$-dynamical system is the finite orbit $f(a),f^2(a),\dots,f^k(a)$ together with the family of balls $\mathcal B_{f^s(a)}(a)=\{x:|x-a|<f^s(a)\}$.  Viewing orbit terms as radii and $a$ as a center produces an elementary geometric language that (i) separates \emph{inflation} and \emph{deflation} phases of an orbit, (ii) allows one to define wave–type invariants (wavelets, amplitude, frequency, total wave), and (iii) reduces convergence questions to quantitative statements about the ''random'' part of the wave decomposition. The main technical results quantify these reductions: a restriction law that bounds the regular contribution to the total wave, equivalence criteria linking finiteness of the wave frequency to convergence of the ball sequence, and the wave-estimate identities that relate amplitude, frequency, and accumulated wave. These estimates are elementary and robust and designed to be applicable to a broad class of arithmetic iteration problems beyond the classical Collatz map.\\

A striking byproduct of this framework is a concrete connection between the backward Collatz process and classical questions in prime distribution (in particular, Sophie Germain primes). After unit left-translates of backward preimage sequences are considered, the algebraic relations between successive backward terms produce conditions that, when coupled with primality, force Sophie Germain-type patterns. Although this observation does not resolve the prime distribution problems, it reframes them on the much thinner and more structured sets furnished by inverse Collatz trees; we believe this reframing may be a useful vantage point for future conditional and heuristic investigations.

\subsection{Organization of the paper}
The technical narrative of the paper proceeds as follows.  In Section \ref{sec:collatz-process}, we introduce a slightly modified Collatz map and develop the Collatz process in detail: definitions, elementary propositions (uniqueness and parity facts for generators) and simple number-theoretic consequences (relations between index and order, residue class obstructions).  Section \ref{sec:dynamical-balls} presents the dynamical-ball formalism: definitions, basic topological and combinatorial properties, the notion of dynamical waves, and the principal structural results (restriction law, wave decomposition, and wave estimates). Here, we collect applications to the Collatz map: reinterpreting classical stopping-time heuristics in the ball language, formulating equivalent variations of the Collatz conjecture, and exposing the connection to Sophie Germain primes via the backward process. We finish with a discussion of natural extensions (translation and dilation principles for dynamical balls) and several concrete conjectures that isolate the critical arithmetic bottlenecks for making the method effective.\\

To summarize, the paper offers (i) a compact, algebraic rephrasing of Collatz orbits through the Collatz process, (ii) a geometric-combinatorial toolkit of dynamical balls and waves that reduces convergence to explicit quantitative constraints, and (iii) novel structural links between inverse Collatz trees and prime patterns.  The methods are elementary in spirit but flexible; we expect that they will serve both as a conceptual language for organizing future computational evidence and as a seedbed for more refined analytic or sieve-type attacks on orbit structure.\\

Recall the Collatz function, the arithmetic function of the form 

\begin{definition}
Let $a\in \mathbb{N}$. The Collatz function is the piece-wise function 
\begin{align}
\mathcal{C}(a)=
\begin{cases}\frac{a}{2} \quad \text{if} \quad a\equiv 0\pmod 2\\3a+1 \quad \text{if} \quad a\equiv 1\pmod 2.
\end{cases}\nonumber
\end{align}
\end{definition}

The Collatz conjecture, one of the acclaimed hardest but easy to state problems is the assertion

\begin{conjecture}
Let $\mathcal{C}$ be the Collatz function. We have $\mathrm{min}\{\mathcal{C}^{s}(b)\}_{s=0}^{\infty}=1$ for any $b\in \mathbb{N}$.
\end{conjecture}

The conjecture has long been studied, hence the vast literature and surveys concerning its study. For example, the problem has received fair treatment in the following surveys \cite{lagarias2010ultimate}, \cite{lagarias19853}, \cite{chamberland2003update}. Motivated by this problem, we introduce the subject of the Collatz process. We develop this theory in much more detail. It turns out incidentally that it is connected to other open problems, such as the problem concerning the distribution of the Sophie Germain primes.\\ 

It needs to be said that the classical problem of deciding on the convergence of a given sequence is generally in principle not a hard problem. However, the difficulty may arise from how the terms in the sequence are generated. A typical example of a sequence whose convergence may be difficult to determine is the Collatz sequence

\begin{align}
f(n),f^2(n),f^3(n),\ldots, f^{k}(n),\ldots \nonumber
\end{align}
where $f^s=f^{s-1}\circ f=f\circ f^{s-1}$ and 
\begin{align}
f(n):=
\begin{cases}
\frac{n}{2} \quad \text{if} \quad n\equiv 0\pmod 2\\3n+1 \quad \text{if} \quad n\equiv 1\pmod 2.
\end{cases}\nonumber
\end{align}
The Collatz conjecture \cite{guy2004unsolved} is the problem that asks us to determine the convergence of the system for all $n\in \mathbb{N}$. Another problem of possibly similar difficulty is the problem of determining the convergence of the juggler sequence introduced by Pickover (see, e.g, \cite{pickover1991computers})

\begin{align}
g(n),g^2(n),g^3(n),\ldots,g^k(n),\ldots \nonumber
\end{align}
with $g^s=g\circ g^{s-1}=g^{s-1}\circ g=g\circ g\cdots \circ g$~($s$ rimes) for all $n\in \mathbb{N}$, where 
\begin{align}
g(n):=
\begin{cases}
\lfloor n^{\frac{1}{2}}\rfloor \quad \text{if} \quad n\equiv 0\pmod 2\\ \lfloor n^{\frac{3}{2}}\rfloor \quad \text{if} \quad n\equiv 1\pmod 2.
\end{cases}\nonumber
\end{align}

These problems are widely believed to be difficult and forbidden, given that there is currently no viable tool to make considerable progress  \cite{guy2004unsolved}, \cite{lagarias2010ultimate}.\\ 

In this paper, we generalize these problems by introducing the notion of dynamical systems and their corresponding dynamical balls. We develop some general tools for studying problems of this form.

\section{Modified Collatz function and the Collatz process}\label{sec:collatz-process}

In this section, we introduce a slight variant of the Collatz function and introduce the notion of the Collatz process. We introduce the notion of the backward Collatz process and the generator of the Collatz process.

\begin{definition}\label{collatz}
Let $a\geq 1$. The Collatz function is the piece-wise function \begin{align}
f(a)=
\begin{cases}\frac{a}{2} \quad \text{if} \quad a\equiv 0\pmod 2,~a>1\\3a+1 \quad \text{if} \quad a\equiv 1\pmod 2,~a>1\\1\quad \text{If} \quad a=1.
\end{cases}\nonumber
\end{align}
\end{definition}
\bigskip

\begin{definition}
Let $f$ be the Collatz function, and let $a\geq 1$. By the \emph{Collatz process} on $a$, we mean the sequence  $\{f^s(a)\}_{s=1}^{\infty}$ for $s\in \mathbb{N}$. The sequence $\{\mathrm{Inf}\{f^{-s}(a)\}\}_{s=1}^{\infty}$ is the \emph{backward Collatz process} and  $a$ is said to be the \emph{generator} of the Collatz process if each $a_n\in \{\mathrm{Inf}\{f^{-s}(a)\}\}_{s=1}^{\infty}$ is of the same parity.
\end{definition}
\bigskip

\begin{proposition}\label{generator}
Let $f$ be the Collatz function with the corresponding Collatz process $\{f^{s}(b)\}_{s=1}^{\infty}$. If $b$ is the generator, then each $a_{n}\in \{\mathrm{Inf}\{f^{-s}(b)\}\}_{s=1}^{\infty}$ must satisfy
\begin{align}
a_n\equiv 0\pmod 2.\nonumber
\end{align}
\end{proposition}

\begin{proof}
Let $f$ be the Collatz function with generator $b\in \mathbb{N}$. First, we observe that $f^{-1}(b)\not \equiv 1\pmod 2$; Otherwise, it would mean $f^{-2}(b)\equiv 0\pmod 2$ contradicting the assumption that $b$ is the generator of the process. It suffices to consider the case $m\geq 2$. Suppose for the sake of contradiction that there exists some $a_m\in \{\mathrm{Inf}\{f^{-s}(b)\}\}_{s=1}^{\infty}$; in particular, let $a_m=f^{-m}(b)$ be such that $a_m\equiv 1\pmod 2$ for $m\geq 2$. Under the Collatz process, we must have 
\begin{align}
3f^{-m}(b)&=f^{-(m-1)}(b)-1.\nonumber
\end{align}
It follows that $f^{-(m-1)}(b)\equiv 0\pmod 2$, thus contradicting the underlying assumption that $b$ is the generator of the process.
\end{proof}
\bigskip

\begin{proposition}\label{absent}
Let $f$ be the Collatz function with the corresponding Collatz process $\{f^s(b)\}_{s=1}^{\infty}$. If $b\in\mathbb{N}\setminus \{1\}$ is the generator of the process, then $b\notin \{f^s(b)\}_{s=1}^{\infty}$.
\end{proposition}

\begin{proof}
Let $f$ be the Collatz function with the corresponding process $\{f^s(b)\}_{s=1}^{\infty}$ with generator $b\in \mathbb{N} \setminus \{1\}$. Suppose, on the contrary, that $b\in \{f^s(b)\}_{s=1}^{\infty}$. Then it follows that there exist some $s\geq 1$ such that $f^s(b)=b$. Thus, we obtain the following chains of equality 
\begin{align}
b=f^{s}(b)=f^{2s}(b)=f^{3s}(b)=\ldots \nonumber
\end{align}
for some $s\geq 1$. It follows from Proposition \ref{generator} that each $a_n\in \{f^{s}(b)\}_{s=m;m\geq 1}^{\infty}$ must satisfy the parity condition $a_n\equiv 0\pmod 2$, since $b$ is the generator of the process. This is not true since $f$ is the Collatz function.
\end{proof}
\bigskip

\begin{remark}
Next we prove the unicity of generators of the Collatz process.
\end{remark}

\begin{proposition}\label{unique}
The generator of any Collatz process is unique.
\end{proposition}

\begin{proof}
Let $f$ be the Collatz function with the corresponding processes $\{f^s(a)\}_{s=1}^{\infty}=\{f^s(b)\}_{s=1}^{\infty}$, where $a,b\in \mathbb{N}$ are the two generators such that $a\neq b$ with $a,b>1$. Then it follows that for any $m\geq 1$, there exists some $r>1$ such that we have $f^{m}(b)=f^r(a)$. Without loss of generality, we let $m>r$ so that we have $f^{m-r}(b)=a$. It follows that $a\in \{f^{s}(b)\}_{s=1}^{\infty}$. By Proposition \ref{absent}, we have $a\notin \{f^{s}(a)\}_{s=1}^{\infty}$ and it follows that $\{f^{s}(a)\}_{s=1}^{\infty}\neq \{f^{s}(b)\}_{s=1}^{\infty}$, contradicting the assumption that $a,b$ are two distinct generators of the process.
\end{proof}

It is very important to note that a Collatz process may or may not have a generator. If a Collatz process has a generator, then we say that the generator is finite; On the other hand, if it has no generator, then we say that the generator is infinite. Next, we expose the parity of a generator of a Collatz process.

\begin{proposition}\label{gparity}
Let $f$ be the Collatz function with the corresponding process $\{f^{s}(b)\}_{s=1}^{\infty}$ for $b\in \mathbb{N}$. If $b\neq 1$ is the generator of the process, then $b\equiv 1\pmod 2$.
\end{proposition}

\begin{proof}
Let $f$ be the Collatz function with the corresponding process $\{f^{s}(b)\}_{s=1}^{\infty}$ and suppose that $b\in \mathbb{N}$ is the generator of the process. By definition, each $a_n\in \{\mathrm{Inf}\{f^{-s}(b)\}\}_{s=1}^{\infty}$ has the same parity and must satisfy $a_n\equiv 0\pmod 2$. Suppose, on the contrary, that $b\equiv 0\pmod 2$, then we choose $m\geq 1$ for $m=\mathrm{Inf}(s)_{s=1}^{\infty}$ such that $f^m(b)\equiv 1 \pmod 2$. It follows that each 
\begin{align}
a_n\in \{\mathrm{Inf}\{f^{-s}(b)\}\}_{s=1}^{\infty}\cup \{b, f(b),f^2(b),\ldots,f^{m-1}(b)\}\nonumber 
\end{align}
has the same parity. It follows that $f^{m}(b)$ is the generator of the process $\{f^{s}(b)\}_{s=m+1;m\geq 1}^{\infty}$. Since $\{f^{s}(b)\}_{s=m+1;m\geq 1}^{\infty}\subset \{f^{s}(b)\}_{s=1}^{\infty}$, it follows that for some $f^{r}(b)\in \{f^{s}(b)\}_{s=m+1;m\geq 1}^{\infty}$, there exist some $f^{t}(b)\in \{f^{s}(b)\}_{s=1}^{\infty}$ such that $f^{r}(b)=f^{t}(b)$. It follows that $f^{k}(b)=b$ for $k\geq 1$. This relation is absurd under the Collatz function.
\end{proof}

\subsection{The order and index under the Collatz process}

In this section, we introduce the notion of the \emph{order} and the \emph{index} of positive integers under the Collatz process. We study the convergence and the divergence of the Collatz process. We launch the following terminology to aid our inquiry.

\begin{definition}\label{process}
Let $f$ be the Collatz function and $a>1$. The \emph{order} of $a$ under the Collatz process is the least value of $m$ such that $f^{m}(a)=2^{k}$. The value of $k$ is the \emph{index} of $a$ under the Collatz process. The number $a$ is said to have finite order and a finite index if and only if it converges under the Collatz process. Otherwise, we say it diverges under the Collatz process. We denote by $\tau_f(a)$ and $\mathrm{Ind}_f(a)$ the \emph{period} and the \emph{index} of $a$ under the Collatz process. In the case $a$ diverges under the process, then $\tau_f(a)=\infty$ and $\mathrm{Ind}_f(a)=\infty$. 
\end{definition}
\bigskip

In light of definition \ref{process}, the Collatz conjecture can be restated in the following manner:

\begin{conjecture}[Collatz]
Let $f$ be the Collatz function and $\{f^{s}(a)\}_{s=1}^{\infty}$ for $a\in \mathbb{N}$ be a Collatz process. We have $\tau_{f}(a)<\infty$.
\end{conjecture}
\bigskip

The above conjecture can also be expressed in a more quantitative form. In other words, it suffice to resolve the Collatz conjecture by showing that 

\begin{conjecture}[Collatz]
Let $f$ be the Collatz function and $\{f^{s}(b)\}_{s=1}^{\infty}$ an arbitrary Collatz process. We have
\begin{align}
\sum \limits_{s=1}^{\infty}\log(f^s(b))<\infty.\nonumber
\end{align}
\end{conjecture}
\bigskip

\begin{proposition}
Let $f$ be the Collatz function and let $a>1$ with $\Omega(a)=2$ such that $a\equiv 0\pmod 2$, then 
\begin{align}
f(r)-1=3f(a)\nonumber
\end{align}
where $a=2r$ with $r\equiv 1\pmod 2$.
\end{proposition}

\begin{proof}
Since $a$ is even, it follows from definition \ref{collatz} that the right hand side must be $3r$. Under the condition that $\Omega(a)=2$ with $r\equiv 1\pmod 2$, the result follows by definition \ref{collatz}.
\end{proof}

\begin{remark}
Next, we show that primes in a certain congruence class should, by necessity, have large order in as much as their index under the Collatz process is large.
\end{remark}

\begin{theorem}\label{residue}
Let $f$ be the Collatz function and $p>3$ be a prime such that $p\equiv 3\pmod 4$. If $\mathrm{Ind}_f(p)>1$, then $\tau_f(p)>1$.
\end{theorem}

\begin{proof}
Let $p>3$ be a prime, then under the Collatz process \ref{process}, it follows that $f^{\tau_f(p)}(p)=2^{\mathrm{Ind}_f(p)}$. Suppose, on the contrary, that $\tau_f(p)=1$. Then under the assumption $\mathrm{Ind}_f(p)>1$, it follows that $2^{\mathrm{Ind}_f(p)}+1\equiv 1\pmod 4$. It must be that $2^{\mathrm{Ind}_f(p)}-1\equiv 3\pmod 4$, so that under the Collatz process we have 
\begin{align}
3p &\equiv 3\pmod 4 \nonumber \\&\Longleftrightarrow p\equiv 1\pmod 4\nonumber
\end{align}
thereby contradicting the residue class of the prime $p>3$.  
\end{proof}

\begin{remark}
Next, we establish a converse of Theorem \ref{residue} in the following proposition.
\end{remark}

\begin{proposition}\label{converse}
Let $f$ be the Collatz function with the corresponding convergent Collatz process $\{f^{s}(b)\}_{s=1}^{\infty}$ for $b\in \mathbb{N}$. If $\tau_f(b)\geq 2$, then $\mathrm{Ind}_f(b)>1$.
\end{proposition}

\begin{proof}
Let $f$ be the Collatz function with a corresponding convergent Collatz process $\{f^{s}(b)\}_{s=1}^{\infty}$. Let $\tau_f(b)\geq 2$ and suppose, on the contrary, that $\mathrm{Ind}_f(b)=1$, then we can write $f^{\tau_f(b)}(b)=2$. Since $\tau_f(b)\geq 2$, we can write $f^{\tau_f(b)-1}(b)=f^{-1}(2)=4$. This contradicts the minimality of $\tau_f(b)$, since $f^{\tau_f(b)-1}(b)\in \{f^{s}(b)\}_{s=1}^{\infty}$.
\end{proof}
\bigskip

\subsection{Relative speed of the Collatz process}
In this section, we introduce the notion of the relative \emph{speed} of a Collatz process.

\begin{definition}\label{relative}
Let $f$ be the Collatz function with the corresponding Collatz process $\{f^{s}(a)\}_{s=1}^{\infty}$. By the speed of the $j^{th}$ Collatz process relative to the $k^{th}$ Collatz process, we mean the expression 
\begin{align}
\nu(f^j(a),f^k(a))=\frac{|f^{k}(a)-f^{j}(a)|}{|k-j|}.\nonumber
\end{align}
\end{definition}
\bigskip

The Collatz conjecture can also be framed in the language of the relative speed of the Collatz process as 

\begin{conjecture}
Let $f$ be the Collatz function with the corresponding Collatz process $\{f^{s}(b)\}_{s=1}^{\infty}$ for $b\in \mathbb{N}$. There exist some $1\leq j<k$ such that $\nu(f^{j}(b),f^{k}(b))=2^r$ for some $r\in \mathbb{N}$.
\end{conjecture}
\bigskip

It follows from definition \ref{relative} that the relative speed of the Collatz process must satisfy the inequality 
\begin{align}
|f^{k}(b)-f^{j}(b)|=|k-j|\nu(f^{j}(b),f^{k}(b))\leq f^{j}(b)+f^{k}(b).\nonumber
\end{align}
Thus the Collatz conjecture is equivalent to establishing the inequality 

\begin{conjecture}\label{Collatz relative}
Let $f$ be the Collatz function with the corresponding process $\{f^{s}(b)\}_{s=1}^{\infty}$. There exists some $k\geq 1$ such that the inequality is valid 
\begin{align}
2^{r}\leq \nu(f^{k+1}(b),f^{k}(b))\leq 2^{m}\nonumber
\end{align}
for some $m,r\in \mathbb{N}$.
\end{conjecture}

\subsection{The sub-Collatz process}

In this section, we introduce the notion of the \emph{sub-Collatz process}. We establish a relationship between the order and the index of a number under the assumption that the Collatz process converges. We launch the following language.

\begin{definition}\label{subcollatz}
Let $f$ be the Collatz function. The Collatz process $\{f^t(a)\}_{t=1}^{\infty}$ is said to be a sub-Collatz process of the Collatz process $\{f^{s}(b)\}_{s=1}^{\infty}$ for $s,t\in \mathbb{N}$ if $\{f^t(a)\}_{t=1}^{\infty}\subseteq \{f^{s}(b)\}_{s=1}^{\infty}$. It is said to be \emph{proper} if $\{f^t(a)\}_{t=1}^{\infty}\subset \{f^{s}(b)\}_{s=1}^{\infty}$. The Collatz process $\{f^{s}(b)\}_{s=1}^{\infty}$ is said to be full if $b$ is the generator of the process.
\end{definition}
\bigskip

\begin{remark}
Next, we state a result that indicates that Collatz processes are indistinguishable once they overlap.
\end{remark}

\begin{proposition}
Let $\{f^{s}(a)\}_{s=1}^{\infty}$ and $\{f^{s}(b)\}_{s=1}^{\infty}$ be full Collatz processes. If $\{f^{s}(a)\}_{s=1}^{\infty}\cap \{f^{s}(b)\}_{s=1}^{\infty}\neq \emptyset$, then $\{f^{s}(a)\}_{s=1}^{\infty}=\{f^{s}(b)\}_{s=1}^{\infty}$ and $a=b$.
\end{proposition}

\begin{proof}
Let $\{f^{s}(a)\}_{s=1}^{\infty}$ and $\{f^{s}(b)\}_{s=1}^{\infty}$ be full Collatz processes and suppose $\{f^{s}(a)\}_{s=1}^{\infty}\cap \{f^{s}(b)\}_{s=1}^{\infty}\neq \emptyset$. It follows that there exist some $t, m\geq 1$ such that $f^t(a)=f^m(b)$. Thus, we obtain the following chains of equality \begin{align}
f^{(t+1)}(a)=f^{(m+1)}(b), \quad f^{(t+2)}(a)=f^{(m+2)}(b), \ldots f^{t+j}(a)=f^{(m+j)}(b) \nonumber 
\end{align}
for all $t\geq 1$. It follows that $\{f^{s}(a)\}_{s=t}^{\infty}=\{f^{s}(b)\}_{s=m}^{\infty}$. It follows that $\{f^{s}(a)\}_{s=1}^{\infty}\subseteq \{f^{s}(b)\}_{s=1}^{\infty}$ and $\{f^{s}(b)\}_{s=1}^{\infty}\subseteq \{f^{s}(a)\}_{s=1}^{\infty}$. This implies that $\{f^{s}(a)\}_{s=1}^{\infty}=\{f^{s}(b)\}_{s=1}^{\infty}$. Since the processes $\{f^{s}(a)\}_{s=1}^{\infty}$ and $\{f^{s}(b)\}_{s=1}^{\infty}$ are full, it follows by definition \ref{subcollatz} that $a$ and $b$ are two generators of the process. Using Proposition \ref{unique}, it follows that $a=b$.
\end{proof}

\begin{proposition}\label{Merssene}
Let $f$ be the Collatz function with the corresponding full process $\{f^{s}(b)\}_{s=1}^{\infty}$. If the generator is trivial, then each $a_n\in \{\mathrm{Inf}\{f^{-s}(b)\}-1\}$ must be of the form $c_n=2^{n}-1$.
\end{proposition}

\begin{proof}
Let $f$ be the Collatz function and suppose $\{f^{s}(b)\}_{s=1}^{\infty}$ is a full process. Then it follows that $b$ is the generator of the process and $\{\mathrm{Inf}\{f^{s}(b)\}\}_{s=1}^{\infty}$ is the backward Collatz process, so that for each $a_n\in \{\mathrm{Inf}\{f^{s}(b)\}\}_{s=1}^{\infty}$ satisfies the parity condition $a_n\equiv 0\pmod 2$. Since $\frac{f^{m+1}(b)}{2}=f^{-m}(b)$ for $m\geq 1$, it follows that $a_n=2^{n-1}f^{-1}(b)$. Since the process is trivial, It follows that $b=1$ and $f^{-1}(1)=2$ and the result follows immediately.
\end{proof}

There does appear an important relation between the index and the order of primes generators. In light of this, we state the following conjecture

\begin{conjecture}
Let $f$ be the Collatz function and $\{f^{s}(b)\}_{s=1}^{\infty}$ be a full Collatz process. If $\{f^{s}(b)\}_{s=1}^{\infty}$ converges, then $b$ is prime if and only if 
\begin{align}
\mathrm{Ind}_f(b)=\tau_f(b)+1.\nonumber
\end{align}
\end{conjecture}

\subsection{Distribution of the Collatz process}

In this section, we study the local and the global distribution of any Collatz process. We focus our study to the existence of primes in any Collatz process under a given generator.

\begin{proposition}
Let $f$ be the Collatz function with the corresponding full process $\{f^{s}(b)\}_{s=1}^{\infty}$ for $b>1$. If $f^2(b)$ is prime, then $2f^{2}(b)-1$ cannot be prime.
\end{proposition}

\begin{proof}
Let $f$ be the Collatz function with the corresponding full Collatz process $\{f^{s}(b)\}_{s=1}^{\infty}$. Then $b$ is the generator and it follows from Proposition \ref{gparity} that $b\equiv 1\pmod 2$, so that $f(b)\equiv 0\pmod 2$. Then it is certainly the case that $f^{2}(b)=\frac{f(b)}{2}=\frac{3b+1}{2}$. The claim follows from this relation.
\end{proof}

\begin{remark}
Next, we expose the theory to the infamous problem concerning the distribution of the Sophie Germain primes. It reduces the problem entirely to knowing the existence and the distribution of consecutive primes in the unit left translates of the backward Collatz process.
\end{remark}

\begin{theorem}\label{sophie germain}
Let $f$ be the Collatz function with the corresponding Collatz process $\{f^{s}(b)\}_{s=1}^{\infty}$. Let $\{f^{s}(b)\}_{s=1}^{\infty}$ be a full process and $f^{-k}(b),f^{-(k+1)}(b)\in \{\mathrm{Inf}\{f^{-s}(b)\}_{s=1}^{\infty}$, the backward Collatz process. If $f^{-k}(b)-1,f^{-(k+1)}(b)-1$ are both prime, then $f^{-k}(b)-1$ must be a Sophie Germain prime. Moreover, there are infinitely many Sophie Germain primes if there are infinitely many consecutive primes in $\{\mathrm{Inf}\{f^{-s}(b)\}-1\}_{s=1}^{\infty}$. 
\end{theorem}

\begin{proof}
Let $f$ be the Collatz function with the corresponding Collatz process $\{f^{s}(b)\}_{s=1}^{\infty}$. Since the process $\{f^{s}(b)\}_{s=1}^{\infty}$ is full, it follows that $b$ must be the generator of the process. It follows from Proposition \ref{generator} that each $f^{-m}(b)\in \{\mathrm{Inf}\{f^{-s}(b)\}_{s=1}^{\infty}$ must satisfy the parity condition $f^{-m}(b)\equiv 0\pmod 2$. Under the Collatz process, we obtain the following increasing sequence 
\begin{align}
\ldots >f^{-(m+2)}(b)>f^{-(m+1)}(b)>f^{-m}(b)>\ldots >f^{-1}(b)\nonumber
\end{align}
with each term satisfying the equality $f^{-(j+1)}(b)=2f^{-j}(b)\Longleftrightarrow f^{-(j+1)}(b)-1=2(f^{-j}(b)-1)+1$. Under the assumption that $f^{-k}(b)-1,f^{-(k+1)}(b)-1$ are both prime, the the result follows immediately. If there are infinitely many consecutive primes $f^{-k}(b)-1,f^{-(k+1)}(b)-1\in \{\mathrm{Inf}\{f^{-s}(b)\}-1\}_{s=1}^{\infty}$, then it will follow that there are infinitely many Sophie Germain primes.
\end{proof}
\bigskip

Theorem \ref{sophie germain} relates the Collatz process to the problem of the distribution of Sophie Germain primes. Indeed, it sets out the idea that studying the problem on the set of integers is in some way silly. Instead, it will much more technique convenient to study the unit left translate of the sequence arising from the backward Collatz process.

\begin{remark}
It turns out that certain Collatz process when subject to a unit left translate has most of its elements being prime. This notion is exemplified in the following result. 
\end{remark}

\begin{theorem}
Let $f$ be a Collatz function, with the corresponding Collatz process $\{f^{s}(b)\}_{s=1}^{\infty}$. If the process is full, then $\mathcal{M}=\{\mathrm{Inf}\{f^{-s}(b)\}-1\}_{s=1}^{\infty}$ contains a prime.
\end{theorem}

\begin{proof}
Let $f$ be the Collatz function with the corresponding full process $\{f^{-s}(b)\}_{s=1}^{\infty}$. It follows that $b$ is the generator of the process. It follows that the sequence $\{\mathrm{Inf}\{f^{-s}(b)\}_{s=1}^{\infty}$ is the backward Collatz process. By Proposition \ref{gparity}, each $a_n\in \{\mathrm{Inf}\{f^{-s}(b)\}_{s=1}^{\infty}$ must satisfy the parity condition $a_n\equiv 0\pmod 2$ and additionally that 
\begin{align}
a_n=2^{n-1}f^{-1}(b)-1\nonumber
\end{align}
for $n\geq 1$ with $f^{-1}(b)\in \{\mathrm{Inf}\{f^{-s}\}_{s=1}^{\infty}$. It follows that $a_n$ must be prime for some $n\geq 1$.
\end{proof}

\begin{conjecture}\label{germain}
Let $f$ be a Collatz function, with the corresponding Collatz process $\{f^{s}(b)\}_{s=1}^{\infty}$. If the process is full, then $\mathcal{M}=\{\mathrm{Inf}\{f^{-s}(b)\}-1\}_{s=1}^{\infty}$ contains infinitely consecutive primes and 
\begin{align}
\lim \limits_{s\longrightarrow \infty}\frac{\#(\mathcal{M}\cap \rho)}{\# \mathcal{M}}=1\nonumber
\end{align}
where $\rho$ is the set of all primes.
\end{conjecture}
\bigskip

Conjecture \ref{germain} is equivalent to the problem of deciding on the distribution of the Sophie Germain prime, which is still an unsolved problem in the subject and we do not pursue in this paper. The Collatz process with a given generator exhibits other stunning and subtle properties in terms of the primality of the terms. In light of this we make the following conjectures

\begin{conjecture}
Let $p>2$ be prime with the corresponding Collatz process such that $p\equiv 3\pmod 4$, then there exist $n\in \{f^{s}(p)\}_{s=1}^{\infty}$ such that $\mu(n)\neq 0$.
\end{conjecture} 
\bigskip

\begin{conjecture}
If $\{f^{s}(b)\}_{s=1}^{\infty}$ is a full Collatz process, then there exists an odd prime $p\in \{f^{s}(b)\}_{s=1}^{\infty}$.
\end{conjecture}
\bigskip

\begin{conjecture}
Let $f$ be the Collatz function with the corresponding Collatz process $\{f^{s}(a)\}_{s=1}^{\infty}$. If the process $\{f^{s}(a)\}_{s=1}^{\infty}$ is full, then $\Omega(a)\leq 2$.
\end{conjecture}
\bigskip 

\section{Dynamical systems induced by sequences}\label{sec:dynamical-balls}

In this section, we introduce and develop the notion of \emph{dynamical systems} induced by fixed $a\in \mathbb{N}$ and their associated induced \emph{dynamical balls}. We develop tools to study problems that require determining the convergence of certain sequences generated by iterating on a fixed integer.

\begin{definition}
Let $f:\mathbb{N}\longrightarrow \mathbb{N}$. By the first $k$ \emph{dynamical system} induced by $f$ on $a\in \mathbb{N}$, we mean the sequences generated by the system of iterations 
\begin{align}
f(a),f^2(a),f^3(a),\ldots,f^{k}(a)\nonumber
\end{align}
where $f^{s}(a)=f\circ f^{s-1}(a)$ with $f^{0}(a)=a$ and equipped with self generative energy
\begin{align}
\mathcal{E}_a(f,f^2,\ldots,f^k):=\prod \limits_{i=1}^{k}f^i(a)\nonumber
\end{align}
with corresponding sequence of balls 
\begin{align}
\mathcal{B}_{f(a)}(a), \mathcal{B}_{f^2(a)}(a),\ldots,\mathcal{B}_{f^k(a)}(a)\nonumber
\end{align}
where $f^s(a)$ is the \emph{radius} of the $s^{th}$ ball in the sequence and $a$ the \emph{center} of each ball. We call each ball in the sequence generated in this manner a \emph{dynamical} ball. In other words, we say that $f$ induces a $k$ dynamical system on $a\in \mathbb{N}$. We call $f^s(a)$ for $1\leq s\leq k$ the $s^{th}$ dynamical system and $\mathcal{B}_{f^s(a)}(a)$ the $s^{th}$ dynamical ball. We say that the $s^{th}$ dynamical system has an upward \emph{measure} relative to the $(s-1)^{th}$ dynamical system if $f^{s}(a)>f^{s-1}(a)$. On the other hand, we say that it has a downward \emph{measure} relative to the $(s-1)^{th}$ dynamical system if $f^{s-1}(a)>f^s(a)$. Similarly, we say that the $s^{th}$ dynamical ball $\mathcal{B}_{f^s(a)}(a)$ is \emph{inflated} relative to the ball $\mathcal{B}_{f^{s-1}(a)}(a)$ if $f^{s}(a)>f^{s-1}(a)$. We say that it is \emph{deflated} relative to the dynamical ball $\mathcal{B}_{f^{s-1}(a)}(a)$ if $f^{s-1}(a)>f^s(a)$. In the situation where $f^{s}(a)=f^{s-1}(a)$,  we say that the $s^{th}$ dynamical system is \emph{stable} relative to the $(s-1)^{th}$ dynamical system, and the $s^{th}$ dynamical ball $\mathcal{B}_{f^s(a)}(a)$ is \emph{stable} relative to the $(s-1)^{th}$ dynamical ball $\mathcal{B}_{f^{s-1}(a)}(a)$.
\end{definition}

\begin{proposition}\label{basic properties}
Let $f$ and $g$ induce $k$ dynamical system on $a\in \mathbb{N}$. Then
\begin{align}
\mathcal{E}_a(f,f^2,\ldots,f^k)=\mathcal{E}_a(g,g^2,\ldots,g^k)\nonumber
\end{align}
if and only if there exists some permutation $\sigma:[1,k]\longrightarrow [1,k]$ such that $g^{i}(a)=f^{\sigma(j)}(a)$ for any $1\leq i\leq j\leq k$.
\end{proposition}

\begin{proposition}
Let $f$ induce a dynamical system on $a\in \mathbb{N}$. If the $s^{th}$ dynamical system is stable relative to the $(s-1)^{th}$ dynamical system, then 
\begin{align}
f^{s}(a)=f^{s+1}(a)=\ldots =f^{s+l}(a)=\ldots \nonumber
\end{align}
for all $l\geq 1$.
\end{proposition}

\begin{proof}
Suppose that the $s^{th}$ dynamical system is stable relative to the $(s-1)^{th}$ dynamical system, then $f^{s}(a)=f^{s-1}(a)$ so that we have the chain of equality
\begin{align}
f^{s}(a)=f^{s+1}(a)=f^{s+2}(a)=\ldots =f^{s+3}(a)=\ldots \nonumber
\end{align}
by iteration.
\end{proof}

\subsection{Analysis on dynamical balls}

In this section, we develop some topology of dynamical balls and study their interaction with each other. We launch the following languages in the sequel.

\begin{definition}\label{dynamical ball intepretation}
Let $\mathcal{B}_{f^j(a)}(a)$ be the $j$-dynamical ball induced by the $k$-dynamical system $f(a),f^2(a),f^3(a),\ldots,f^{k}(a)$ with $1\leq j\leq k$. We say that $y_n\in \mathcal{B}_{f^j(a)}(a)$ if and only if the inequality holds
\begin{align}
|y_n-a|<f^j(a).\nonumber
\end{align}
In particular, we write $y_n \inn \mathcal{B}_{f^j(a)}(a)$ if and only if $|y_n-a|=f^j(a)$.
\end{definition}

\begin{proposition}
Let $x_n\in \mathcal{B}_{f^j(a)}(a)$ and $y_n\in \mathcal{B}_{f^i(a)}(a)$ with $1\leq i<j\leq k$. Then the following holds
\begin{enumerate}
\item [(i)] $x_n-y_n \in \mathcal{B}_{f^i(af^s(a))}(a)$ provided $f^i(a)+f^i(f^s(a))\leq f^i(af^s(a))$ for $s=j-i$

\item [(ii)] $x_n-y_n \in \mathcal{B}_{f^i(a+f^s(a))}(a)$ provided $f^i(a)+f^i(f^s(a))\leq f^i(a+f^s(a))$ for $s=j-i$.
\end{enumerate}
\end{proposition}

\begin{proof}
The following containment $x_n\in \mathcal{B}_{f^j(a)}(a)$ and $y_n\in \mathcal{B}_{f^i(a)}(a)$ implies that $|x_n-a|\leq f^j(a)$ and $|y_n-a|\leq f^i(a)$ so that with $1\leq i<j\leq k$, we can write the inequality 
\begin{align}
|x_n-y_n-a|\leq f^j(a)+f^i(a)=f^i(f^s(a))+f^i(a)\nonumber
\end{align}
and $(i)$ and $(ii)$ follows under the specified requirements.
\end{proof}

\begin{proposition}\label{dynamical balls non-overlapping}
Let
\begin{align}
f(a),f^2(a),f^3(a),\ldots,f^{k}(a)\nonumber
\end{align} 
be a $k$-dynamical system with corresponding sequence of dynamical balls 
\begin{align}
\mathcal{B}_{f(a)}(a), \mathcal{B}_{f^2(a)}(a),\ldots,\mathcal{B}_{f^k(a)}(a).\nonumber
\end{align}
Then the following embedding holds
\begin{align}
\mathcal{B}_{f^j(a)}(a) \subseteq \mathcal{B}_{f^{j+1}(a)}(a)\nonumber
\end{align}
if and only $f^j(a)\leq f^{j+1}(a)$.
\end{proposition}
\bigskip

This is an easy consequence of the interpretation of the sequence of dynamical balls all centered at a fixed $a\in \mathbb{N}$ and evolves according to the radius whose values are terms in the corresponding induced dynamical system induced on $a\in \mathbb{N}$ by $f:\mathbb{N}\longrightarrow \mathbb{N}$. 

\subsection{The limit of dynamical balls}

In this subsection, we introduce and study the notion of the limit of dynamical balls.

\begin{definition}\label{limits}
Let
\begin{align}
f(a),f^2(a),f^3(a),\ldots,f^{k}(a)\nonumber
\end{align} 
be a $k$-dynamical system with corresponding sequence of dynamical balls 
\begin{align}
\mathcal{B}_{f(a)}(a), \mathcal{B}_{f^2(a)}(a),\ldots,\mathcal{B}_{f^k(a)}(a).\nonumber
\end{align}
We denote the limit of the dynamical balls by $\lim \limits_{k\longrightarrow \infty}\mathcal{B}_{f^k(a)}(a)$. We write
\begin{align}
\lim \limits_{k\longrightarrow \infty}\mathcal{B}_{f^k(a)}(a)=\mathcal{B}_{b}(a)\nonumber
\end{align} 
if and only if for any $\delta>0$ and for any $x_n\inn \mathcal{B}_{b}(a)$ there exists some $K_o>0$ such that for all $k\geq K_o$ there exists some $y_n\inn \mathcal{B}_{f^k(a)}(a)$ such that 
\begin{align}
|x_n-y_n|<\delta\nonumber
\end{align}
and we say the sequence of $k$-dynamical balls \emph{converges}. Otherwise, we say it \emph{diverges}.
\end{definition}
\bigskip

Let
\begin{align}
f(a),f^2(a),f^3(a),\ldots,f^{k}(a)\nonumber
\end{align} 
be a $k$-dynamical system with corresponding sequence of dynamical balls 
\begin{align}
\mathcal{B}_{f(a)}(a), \mathcal{B}_{f^2(a)}(a),\ldots,\mathcal{B}_{f^k(a)}(a).\nonumber
\end{align}
The following statements are equivalent:\\ 

For any $\epsilon>0$, there exists some $x_n \inn \mathcal{B}_{f^s(a)}(a)$ and $y_n \inn \mathcal{B}_{f^{s-1}(a)}(a)$ such that 
\begin{align}
|x_n-y_n|<\epsilon \nonumber
\end{align}
for $1\leq s\leq k$ as $k\longrightarrow \infty$ if and only if the corresponding infinite dynamical system 
\begin{align}
f(a),f^2(a),f^3(a),\ldots,f^{k}(a),f^{k+1}(a),\ldots \nonumber
\end{align}
is a Cauchy sequence. Let us suppose that the corresponding infinite dynamical system is a Cauchy sequence, then it follows that for any $\epsilon>0$ there exists a $K_o>0$ such that for all $s\geq K_o$
\begin{align}
|f^{s}(a)-f^{s-1}(a)|=|a+f^{s}(a)-(a+f^{s-1}(a))|<\epsilon \nonumber 
\end{align}
so that by taking $x_n=a+f^{s}(a)$ and $y_n=a+f^{s-1}(a)$ we observe that $x_n\inn \mathcal{B}_{f^s(a)}(a)$ and $y_n \inn \mathcal{B}_{f^{s-1}(a)}(a)$. Conversely, suppose that for any $\epsilon>0$ there exists some $x_n \inn \mathcal{B}_{f^s(a)}(a)$ and $y_n \inn \mathcal{B}_{f^{s-1}(a)}(a)$ such that 
\begin{align}
|x_n-y_n|<\epsilon \nonumber
\end{align}
for $1\leq s\leq k$ as $k\longrightarrow \infty$. Taking $x_n=a+f^s(a) \inn \mathcal{B}_{f^s(a)}(a)$ and $y_n=a+f^{s-1}(a) \inn \mathcal{B}_{f^{s-1}(a)}(a)$, we deduce
\begin{align}
|f^s(a)-f^{s-1}(a)|<\epsilon \nonumber
\end{align}
for $1\leq s\leq k$ as $k\longrightarrow \infty$.
\bigskip

\begin{proposition}\label{convergence equivalence}
Let $\mathcal{B}_{f^j(a)}(a)$ be a dynamical ball induced by the $k$-dynamical system $f(a),f^2(a),f^3(a),\ldots,f^{k}(a)$ with $1\leq j\leq k$. Then for any $\epsilon>0$ there exists some $x_n \inn \mathcal{B}_{f^s(a)}(a)$ and $y_n \inn \mathcal{B}_{f^{s-1}(a)}(a)$ such that 
\begin{align}
|x_n-y_n|<\epsilon \nonumber
\end{align}
for $1\leq s\leq k$ as $k\longrightarrow \infty$ if and only if $\lim \limits_{j\longrightarrow \infty}\mathcal{B}_{f^j(a)}(a)$ exists.
\end{proposition}

\begin{proof}
Let $\epsilon>0$ and suppose that there exists some $x_n \inn \mathcal{B}_{f^s(a)}(a)$ and $y_n \inn \mathcal{B}_{f^{s-1}(a)}(a)$ such that 
\begin{align}
|x_n-y_n|<\epsilon \nonumber
\end{align}
for $1\leq s\leq k$ as $k\longrightarrow \infty$. Then it implies that the infinite dynamical system 
\begin{align}
f(a),f^2(a),\ldots,f^k(a),\ldots, \nonumber
\end{align}
must be a Cauchy sequence, so that there exists some $L\in \mathbb{R}^{+}$ such that 
\begin{align}
\lim \limits_{j\longrightarrow \infty}f^j(a)=L.\nonumber
\end{align}
It follows that $\lim \limits_{j\longrightarrow \infty}\mathcal{B}_{f^j(a)}(a)$ exists and 
\begin{align}
\lim \limits_{j\longrightarrow \infty}\mathcal{B}_{f^j(a)}(a)=\mathcal{B}_{L}(a).\nonumber
\end{align}
Conversely, suppose that $\lim \limits_{j\longrightarrow \infty}\mathcal{B}_{f^j(a)}(a)$ exists and let 
\begin{align}
\lim \limits_{j\longrightarrow \infty}\mathcal{B}_{f^j(a)}(a)=\mathcal{B}_{L}(a).\nonumber
\end{align} 
Then it follows that for any $\epsilon>0$ and for any $b\inn \mathcal{B}_{L}(a)$ there exists some $K_o>0$ such that for all $s\geq K_o$ then there exists some $x_n\inn \mathcal{B}_{f^s(a)}(a)$ such that $|x_n-b|<\frac{\epsilon}{2}$. It follows similarly that there exists some $y_n \inn \mathcal{B}_{f^{s-1}(a)}(a)$ such that $|y_n-b|<\frac{\epsilon}{2}$ so that for all $s\geq K_o$
\begin{align}
|x_n-y_n|\leq |x_n-b|+|y_n-b|<\frac{\epsilon}{2}+\frac{\epsilon}{2}=\epsilon.\nonumber
\end{align}
\end{proof}

\subsection{Dynamical waves and amplitude of waves induced by dynamical balls}

In this section, we introduce and study the notion of \emph{dynamical waves} and their corresponding notion of \emph{amplitudes} induced by the evolution of dynamical balls.

\begin{definition}\label{annular radius of dynamical balls}
Let $\mathcal{B}_{f^j(a)}(a)$ be a dynamical ball induced by the $k$-dynamical system $f(a),f^2(a),f^3(a),\ldots,f^{k}(a)$ with $1\leq j\leq k$. We call the sequence of discrepancy 
\begin{align}
(|f^{j+1}(a)-f^j(a)|)_{1\leq j\leq k}\nonumber
\end{align}
the dynamical \emph{waves} induced by the evolution of the dynamical balls. We call each term of the sequence a \emph{wavelet} of the dynamical system. We call $\mathrm{sup}_{1\leq j\leq k}(|f^{j+1}(a)-f^j(a)|)$ the \emph{amplitude} of the wave and we denote amplitude by $\mathbb{A}_f(a,k)$.
\end{definition}
\bigskip

\begin{definition}
Let
\begin{align}
f(a),f^2(a),f^3(a),\ldots,f^{k}(a)\nonumber
\end{align}
be a $k$-dynamical system with corresponding sequence of dynamical balls 
\begin{align}
\mathcal{B}_{f(a)}(a), \mathcal{B}_{f^2(a)}(a),\ldots,\mathcal{B}_{f^k(a)}(a).\nonumber
\end{align}
By the \emph{frequency} of the dynamical \emph{wave} induced, we mean the formal sum 
\begin{align}
\mathcal{W}_a(f,k):=\sum \limits_{j=1}^{k}\frac{|f^{j+1}(a)-f^j(a)|}{j}.\nonumber
\end{align}
We denote the \emph{frequency} of the wave of the corresponding infinite dynamical system as 
\begin{align}
\mathcal{W}_a(f)=\lim_{k\longrightarrow \infty}\sum \limits_{j=1}^{k}\frac{|f^{j+1}(a)-f^j(a)|}{j}=\sum \limits_{j=1}^{\infty}\frac{|f^{j+1}(a)-f^j(a)|}{j}.\nonumber
\end{align}
\end{definition}
\bigskip

It turns out that for any dynamical wave induced by $f$ on $a\in \mathbb{N}$, we can decompose the total dynamical waves into two pieces, namely as small piece and a large piece as follows
\begin{align}
\mathcal{D}_f(a,k):&=\sum \limits_{2\leq s\leq k}|f^s(a)-f^{s-1}(a)| \nonumber \\&=\sum \limits_{\substack{2\leq s\leq k\\|f^s(a)-f^{s-1}(a)|>|f^2(a)-f(a)|}}|f^{s}(a)-f^{s-1}(a)|\nonumber \\&+\sum \limits_{\substack{2\leq s\leq k\\|f^s(a)-f^{s-1}(a)|<|f^2(a)-f(a)|}}|f^{s}(a)-f^{s-1}(a)|.\nonumber
\end{align}
We call the second sum on the right-hand side the \emph{regular} part and the first sum the \emph{random} part. Symbolically, we rewrite the above decomposition into random and regular part as 
\begin{align}
\mathcal{D}_f(a,k):&=\mathbb{R}ad_f(a,k)+\mathbb{R}eg_f(a,k).\nonumber
\end{align}
It is easy to see that for any dynamical system, we can write 
\begin{align}
\mathcal{D}_f(a,k):&=\sum \limits_{\substack{2\leq s\leq k\\|f^s(a)-f^{s-1}(a)|>|f^2(a)-f(a)|}}|f^{s}(a)-f^{s-1}(a)|+O_{f,a}(k).\nonumber
\end{align}
The corresponding total wave of the infinite dynamical system is obtained by taking the limits
\begin{align}
\mathcal{D}_f(a)&=\lim \limits_{k\longrightarrow \infty}\mathcal{D}_f(a,k)\nonumber \\&=\sum \limits_{s=2}^{\infty}|f^s(a)-f^{s-1}(a)|.\nonumber
\end{align}

\begin{proposition}\label{wave-convergence equivalence}
Let
\begin{align}
f(a),f^2(a),f^3(a),\ldots,f^{k}(a)\nonumber
\end{align} 
be a $k$-dynamical system with corresponding sequence of dynamical balls 
\begin{align}
\mathcal{B}_{f(a)}(a), \mathcal{B}_{f^2(a)}(a),\ldots,\mathcal{B}_{f^k(a)}(a).\nonumber
\end{align}
Then $\mathcal{W}_a(f)<\infty$ if and only if $\lim \limits_{j\longrightarrow \infty}\mathcal{B}_{f^j(a)}(a)$ exists.
\end{proposition}

\begin{proof}
Suppose that $\mathcal{W}_a(f)<\infty$. It follows that 
\begin{align}
\sum \limits_{j=1}^{\infty}\frac{|f^{j+1}(a)-f^j(a)|}{j}<\infty \nonumber
\end{align}
so that $\lim \limits_{j\longrightarrow \infty}|f^{j+1}(a)-f^j(a)|=0$ and it implies that $\lim \limits_{j\longrightarrow \infty}f^j(a)$ exists and so is $\lim \limits_{j\longrightarrow \infty}\mathcal{B}_{f^j(a)}(a)$. Conversely, suppose that $\lim \limits_{j\longrightarrow \infty}\mathcal{B}_{f^j(a)}(a)$ exists then so is $\lim \limits_{j\longrightarrow \infty}f^j(a)$ so that  $\lim \limits_{j\longrightarrow \infty}|f^{j+1}(a)-f^j(a)|=0$ and $\mathcal{W}_a(f)<\infty$.
\end{proof}

\begin{theorem}[Restriction law]\label{restriction theorem}
Let
\begin{align}
f(a),f^2(a),f^3(a),\ldots,f^{k}(a)\nonumber
\end{align} 
be a $k$-dynamical system such that $|f^{s+1}(a)-f^{s}(a)|\neq |f^{t+1}(a)-f^t(a)|$ for all $s,t\geq 1$ with $s\neq t$ and $|f^{s+1}(a)-f^{s}(a)|,|f^{t+1}(a)-f^t(a)|\neq 0$ with corresponding sequence of dynamical balls 
\begin{align}
\mathcal{B}_{f(a)}(a), \mathcal{B}_{f^2(a)}(a),\ldots,\mathcal{B}_{f^k(a)}(a).\nonumber
\end{align}
Then $\lim \limits_{k\longrightarrow \infty}\mathbb{R}eg_f(a,k)<\infty$.
\end{theorem}

\begin{proof}
Let us suppose on the contrary that  $\lim \limits_{k\longrightarrow \infty}\mathbb{R}eg_f(a,k)=\infty$ so that 
\begin{align}
\lim \limits_{k\longrightarrow \infty}\sum \limits_{\substack{2\leq s\leq k\\|f^s(a)-f^{s-1}(a)|<|f^2(a)-f(a)|}}|f^{s}(a)-f^{s-1}(a)|\nonumber
\end{align}
contains infinitely many terms. It follows from the condition $|f^s(a)-f^{s-1}(a)|<|f^2(a)-f(a)|$ and the pigeon hole principle that there are infinitely many coinciding wavelets. It follows that there must exist some $s\neq t$ such that $|f^{s+1}(a)-f^{s}(a)|=|f^{t+1}(a)-f^t(a)|$. This contradicts the requirements of the dynamical system induced.
\end{proof}

\begin{remark}
Theorem \ref{restriction theorem} although simple is ultimately useful tool to determine the convergence of dynamical systems. The bound on the regular part of the total wave of any infinite dynamical system reduces the problem of convergence to just the random part of the decomposition. These ideas are summarized in the following proposition.
\end{remark}

\begin{proposition}\label{key theorem}
Let
\begin{align}
f(a),f^2(a),f^3(a),\ldots,f^{k}(a)\nonumber
\end{align} 
be a $k$-dynamical system such that $|f^{s+1}(a)-f^{s}(a)|\neq |f^{t+1}(a)-f^t(a)|$ for all $s,t\geq 1$ with $s\neq t$ and $|f^{s+1}(a)-f^{s}(a)|,|f^{t+1}(a)-f^t(a)|\neq 0$ with corresponding sequence of dynamical balls 
\begin{align}
\mathcal{B}_{f(a)}(a), \mathcal{B}_{f^2(a)}(a),\ldots,\mathcal{B}_{f^k(a)}(a).\nonumber
\end{align}
Then $\lim \limits_{j\longrightarrow \infty}\mathcal{B}_{f^j(a)}(a)$ exists if and only if $\lim \limits_{k \longrightarrow \infty}\mathbb{R}ad_f(a,k)<\infty$.
\end{proposition}

\begin{proof}
Suppose that $\lim \limits_{j\longrightarrow \infty}\mathcal{B}_{f^j(a)}(a)$ exists then it follows that $\lim \limits_{j\longrightarrow \infty}|f^j(a)-f^{j-1}(a)|=0$. It implies that 
\begin{align}
\lim \limits_{k\longrightarrow \infty}\mathbb{R}ad_f(a,k)<\mathcal{D}_f(a)<\infty.\nonumber
\end{align}
Conversely, suppose that $\lim \limits_{k \longrightarrow \infty}\mathbb{R}ad_f(a,k)<\infty$, then by appealing to Theorem \ref{restriction theorem}, we can write
\begin{align}
\lim \limits_{k \longrightarrow \infty}\mathcal{D}_f(a,k)=\lim \limits_{k \longrightarrow \infty}\mathbb{R}ad_f(a,k)+\lim \limits_{k \longrightarrow \infty}\mathbb{R}eg_f(a,k)<\infty \nonumber
\end{align}
so that 
\begin{align}
\sum \limits_{s=1}^{\infty}|f^s(a)-f^{s-1}(a)|<\infty.\nonumber
\end{align}
Since $f^j(a)-f^{j-1}(a)\in \mathbb{Z}$, it implies that $\lim \limits_{j\longrightarrow \infty}|f^j(a)-f^{j-1}(a)|=0$ and that $\lim \limits_{j\longrightarrow \infty}\mathcal{B}_{f^j(a)}(a)$  also exists.
\end{proof}

\subsection{Dynamical waves estimate}

In this section, we establish some new estimates relating the frequency, amplitude and total waves of any dynamical systems. 

\begin{theorem}\label{wave estimates}
Let
\begin{align}
f(a),f^2(a),f^3(a),\ldots,f^{k}(a)\nonumber
\end{align} 
be a $k$-dynamical system such that $|f^{s+1}(a)-f^{s}(a)|\neq |f^{t+1}(a)-f^t(a)|$ for all $s,t\geq 1$ with $s\neq t$ and $|f^{s+1}(a)-f^{s}(a)|,|f^{t+1}(a)-f^t(a)|\neq 0$ with corresponding sequence of dynamical balls 
\begin{align}
\mathcal{B}_{f(a)}(a), \mathcal{B}_{f^2(a)}(a),\ldots,\mathcal{B}_{f^k(a)}(a).\nonumber
\end{align}
Then 
\begin{enumerate}
\item [(i)] $$
            \mathcal{W}_f(a,k)\ll \mathbb{A}_f(a,k)\log k
            $$
\bigskip

\item [(ii)] $$
             \int \limits_{1}^{k-1}\frac{f^t(a)}{t^2}dt\ll_{f,a} \mathcal{W}_f(a,k)-\frac{\mathcal{D}_f(a,k)}{k}
             $$
\bigskip

\item [(iii)] $$
              \mathcal{W}_f(a,k)=\frac{\mathbb{R}ad_f(a,k)}{k}+\int \limits_{1}^{k-1}\frac{\mathbb{R}ad_f(a,t)}{t^2}dt+O(\frac{1}{k})
              $$
\bigskip

\item [(iv)] $$
            \int \limits_{1}^{k-1}\frac{|f^t(a)-f(a)|}{t^2}dt\leq \int \limits_{1}^{k-1}\frac{\mathbb{R}ad_f(a,t)}{t^2}dt+O(\frac{1}{k})
             $$
\end{enumerate}
\end{theorem}

\begin{proof}
\begin{enumerate}

\item [(i)] We can write 
\begin{align}
\mathcal{W}_f(a,k)&=\sum \limits_{j=1}^{k-1}\frac{|f^{j+1}(a)-f^j(a)|}{j}\nonumber \\&\leq \mathbb{A}_f(a,k)\sum \limits_{j=1}^{k-1}\frac{1}{j} \ll \mathbb{A}_f(a,k)\log k.\nonumber
\end{align}
\bigskip

\item [(ii)] By an application of partial summation, we can write the frequency of the dynamical wave
\begin{align}
\mathcal{W}_f(a,k)&=\frac{1}{k}\mathcal{D}_f(a,k)+\int \limits_{1}^{k-1}\frac{\mathcal{D}_f(a,t)}{t^2}dt\nonumber
\end{align}
so that by exploiting the lower bound 
\begin{align}
\mathcal{D}_f(a,t):&=\sum \limits_{1\leq j\leq t}|f^{j+1}(a)-f^j(a)|\geq |f^t(a)-f(a)|\gg_{f,a}f^t(a)\nonumber
\end{align}
the asserted estimate follows.
\bigskip

\item [(iii)] By applying the decomposition of the total dynamical waves into random and the regular part as         $\mathcal{D}_f(a,k)=\mathbb{R}ad_f(a,k)+\mathbb{R}eg_f(a,k)$, we can further write the estimate for the frequency in $(ii)$ as 
\begin{align}
\mathcal{W}_f(a,k)&=\frac{1}{k}\mathcal{D}_f(a,k)+\int \limits_{1}^{k-1}\frac{\mathcal{D}_f(a,t)}{t^2}dt\nonumber \\&=\frac{1}{k}\mathbb{R}ad_f(a,k)+\frac{1}{k}\mathbb{R}eg_f(a,k)+\int \limits_{1}^{k-1}\frac{\mathbb{R}ad_f(a,t)}{t^2}dt+\int \limits_{1}^{k-1}\frac{\mathbb{R}eg_f(a,t)}{t^2}dt\nonumber \\&=\frac{\mathbb{R}ad_f(a,k)}{k}+\int \limits_{1}^{k-1}\frac{\mathbb{R}ad_f(a,t)}{t^2}dt+O(\frac{1}{k})\nonumber
\end{align}
since $$
      \mathbb{R}eg_f(a,k)\ll_{f,a} 1
      $$ 
      and 
      $$
      \int \limits_{1}^{k-1}\frac{\mathbb{R}eg_f(a,t)}{t^2}dt\leq \int \limits_{1}^{\infty}\frac{\mathbb{R}eg_f(a,t)}{t^2}dt \ll_{f,a} \int \limits_{1}^{\infty}\frac{1}{t^2}dt.
      $$
\bigskip

\item [(iv)] By plugging the estimate in $(iii)$ into the upper bound 
\begin{align}
\int \limits_{1}^{k-1}\frac{|f^t(a)-f(a)|}{t^2}dt\leq \mathcal{W}_f(a,k)-\frac{\mathcal{D}_f(a,k)}{k}\nonumber
\end{align} 
the claimed upper bound also follows.
\end{enumerate}
\end{proof}
\bigskip

The estimates established in Theorem \ref{wave estimates} can be used in a unifying manner to study the convergence of any dynamical system. The estimate in $(iii)$ seems to stand out among them and confirms Proposition \ref{key theorem}. We confirm the observation again as an application of the estimate.

\begin{corollary}
Let
\begin{align}
f(a),f^2(a),f^3(a),\ldots,f^{k}(a)\nonumber
\end{align} 
be a $k$-dynamical system such that $|f^{s+1}(a)-f^{s}(a)|\neq |f^{t+1}(a)-f^t(a)|$ for all $s,t\geq 1$ with $s\neq t$ and $|f^{s+1}(a)-f^{s}(a)|,|f^{t+1}(a)-f^t(a)|\neq 0$ with corresponding sequence of dynamical balls 
\begin{align}
\mathcal{B}_{f(a)}(a), \mathcal{B}_{f^2(a)}(a),\ldots,\mathcal{B}_{f^k(a)}(a).\nonumber
\end{align}
Then $\lim \limits_{j\longrightarrow \infty}\mathcal{B}_{f^j(a)}(a)$ exists if and only if $\lim \limits_{k \longrightarrow \infty}\mathbb{R}ad_f(a,k)<\infty$.
\end{corollary}

\begin{proof}
The result follows from the estimate 
\begin{align}
\mathcal{W}_f(a,k)=\frac{\mathbb{R}ad_f(a,k)}{k}+\int \limits_{1}^{k-1}\frac{\mathbb{R}ad_f(a,t)}{t^2}dt+O(\frac{1}{k})\nonumber
\end{align}
and an appeal to Proposition \ref{wave-convergence equivalence}.
\end{proof}

\subsection{Translation and dilation of dynamical balls}

In this section, we introduce the notion of \emph{translation} and \emph{dilation} of dynamical balls. This would allow the movement of dynamical balls for the purposes of our work.

\begin{definition}
Let
\begin{align}
f(a),f^2(a),f^3(a),\ldots,f^{k}(a)\nonumber
\end{align} 
be a $k$-dynamical system with corresponding sequence of dynamical balls 
\begin{align}
\mathcal{B}_{f(a)}(a), \mathcal{B}_{f^2(a)}(a),\ldots,\mathcal{B}_{f^k(a)}(a).\nonumber
\end{align}
We call the map 
\begin{align}
\mathbb{T}_b:\mathcal{B}_{f^j(a)}(a)\longmapsto \mathcal{B}_{f^j(a+b)}(a+b):=\mathcal{B}_{f^j(\mathbb{T}_b(a)}(\mathbb{T}_b(a))\nonumber
\end{align}
the \emph{translation} of the dynamical ball $\mathcal{B}_{f^j(a)}(a)$ by a a scale factor $b$.
\end{definition}
\bigskip

\begin{definition}
Let
\begin{align}
f(a),f^2(a),f^3(a),\ldots,f^{k}(a)\nonumber
\end{align} 
be a $k$-dynamical system with corresponding sequence of dynamical balls 
\begin{align}
\mathcal{B}_{f(a)}(a), \mathcal{B}_{f^2(a)}(a),\ldots,\mathcal{B}_{f^k(a)}(a).\nonumber
\end{align}
We call the map 
\begin{align}
\mathbb{D}_m:\mathcal{B}_{f^j(a)}(a)\longmapsto \mathcal{B}_{f^j(ma)}(ma):=\mathcal{B}_{f^j(\mathbb{D}_m(a)}(\mathbb{D}_m(a))\nonumber
\end{align}
the \emph{dilation} of the dynamical ball $\mathcal{B}_{f^j(a)}(a)$ by a scale factor $m$.
\end{definition}

\begin{proposition}\label{dynamical ball convergence extension via translation}
Let
\begin{align}
f(a),f^2(a),f^3(a),\ldots,f^{k}(a)\nonumber
\end{align} 
be a $k$-dynamical system with corresponding sequence of dynamical balls 
\begin{align}
\mathcal{B}_{f(a)}(a), \mathcal{B}_{f^2(a)}(a),\ldots,\mathcal{B}_{f^k(a)}(a).\nonumber
\end{align}
Suppose that $\lim \limits_{j\longrightarrow \infty}\mathcal{B}_{f^j(b)}(b)$ exists. If $\lim \limits_{j\longrightarrow \infty}\mathcal{B}_{f^j(a)}(a)$ exists then $\lim \limits_{j\longrightarrow \infty}\mathcal{B}_{f^j(a+b)}(a+b)$ exists provided that $f^{s}(a+b)\leq f^s(a)+f^s(b)$ whenever $f^{s-1}(a+b)\geq f^{s-1}(a)+f^{s-1}(b)$ for all $s\geq 2$.
\end{proposition}

\begin{proof}
It suffices to show that for any $\epsilon>0$ there exists some $N_o>0$ such that for all $s\geq N_o$ then $|f^s(a+b)-f^{s-1}(a+b)|<\epsilon$.\\ 

Under the assumption $\lim \limits_{j\longrightarrow \infty}\mathcal{B}_{f^j(b)}(b)$ and $\lim \limits_{j\longrightarrow \infty}\mathcal{B}_{f^j(a)}(a)$ exist, then for any $\epsilon>0$ there exist some $N_o,M_o>0$ such that 
\begin{align}
|f^s(a)-f^{s-1}(a)|<\frac{\epsilon}{2} \nonumber
\end{align}
for all $s\geq N_o$ and 
\begin{align}
|f^s(b)-f^{s-1}(b)|<\frac{\epsilon}{2} \nonumber
\end{align}
for all $s\geq M_o$. Choosing $P=\mathrm{max}\{N_o,M_o\}$ and exploiting the condition $f^{s}(a+b)\leq f^s(a)+f^s(b)$ if $f^{s-1}(a+b)\geq f^{s-1}(a)+f^{s-1}(b)$ for all $s\geq 2$, it follows that 
\begin{align}
|f^s(a+b)-f^{s-1}(a+b)|\leq |f^s(a)-f^{s-1}(a)|+|f^s(b)-f^{s-1}(b)|<\frac{\epsilon}{2}+\frac{\epsilon}{2}=\epsilon \nonumber
\end{align}
for all $s\geq P=\mathrm{max}\{N_o,M_o\}$. This implies that $\lim \limits_{j\longrightarrow \infty}\mathcal{B}_{f^j(a+b)}(a+b)$ exists since $\epsilon>0$ can be chosen arbitrarily.
\end{proof}
\bigskip

\begin{proposition}\label{dynamical ball convergence extension via dilation}
Let
\begin{align}
f(a),f^2(a),f^3(a),\ldots,f^{k}(a)\nonumber
\end{align}
be a $k$-dynamical system with corresponding sequence of dynamical balls 
\begin{align}
\mathcal{B}_{f(a)}(a), \mathcal{B}_{f^2(a)}(a),\ldots,\mathcal{B}_{f^k(a)}(a).\nonumber
\end{align}
If $\lim \limits_{j\longrightarrow \infty}\mathcal{B}_{f^j(a)}(a)$ exists then $\lim \limits_{j\longrightarrow \infty}\mathcal{B}_{f^j(ma)}(ma)$ exists provided that $f^{s}(ma)\leq mf^s(a)$ whenever $f^{s-1}(ma)\geq mf^{s-1}(a)$ for all $s\geq 2$ and for a fixed $m\in \mathbb{N}$.
\end{proposition}

\begin{proof}
Under the assumption $\lim \limits_{j\longrightarrow \infty}\mathcal{B}_{f^j(a)}(a)$ exists, then for any $\epsilon>0$ there exists some $N_o>0$ such that 
\begin{align}
|f^s(a)-f^{s-1}(a)|<\epsilon \nonumber
\end{align}
for all $s\geq N_o$ so that under the conditions $f^{s}(ma)\leq mf^s(a)$ whenever $f^{s-1}(ma)\geq mf^{s-1}(a)$ for all $s\geq 2$ and for a fixed $m\in \mathbb{N}$, we can write by choosing $\epsilon=\frac{\delta}{m}$ for any $\delta>0$
\begin{align}
|f^s(ma)-f^{s-1}(ma)|\leq m|f^s(a)-f^{s-1}(a)|<m\epsilon=\delta \nonumber
\end{align}
for all $s\geq N_o$.
\end{proof}

\begin{remark}
The proposition \ref{dynamical ball convergence extension via translation} and \ref{dynamical ball convergence extension via dilation} provides a slick way to extend the convergence of an infinite dynamical system induced by a function $f$ on any $a\in \mathbb{N}$ to some other numbers $z\in \mathbb{N}$ by translation.
\end{remark}

\footnote{
\par
.}%

\bibliographystyle{amsplain}

\begin{thebibliography}{10}

\bibitem {lagarias2010ultimate} J.C Lagarias, \textit{The ultimate challenge: The 3x+ 1 problem}, American Mathematical Soc., 2010.\\

\bibitem {lagarias19853} J.C Lagarias, \textit{The 3 x+ 1 problem and its generalizations}, The American Mathematical Monthly, vol. 92:1, Taylor \& Francis, 1985, pp 3--23.\\

\bibitem {chamberland2003update} M. Chamberland, \textit{An update on the 3x+ 1 problem}, Butllet{\i} de la Societat Catalana de Matematiques, vol. 18:1, 2003, pp 19--45.\\



\bibitem {guy2004unsolved} R. Guy, \textit{Unsolved problems in number theory}, Springer Science \& Business Media, vol. 1, 2004.\\

\bibitem {pickover1991computers} C.A Pickover, \textit{Computers and the Imagination: Visual Adventures Beyond the Edge}, St. Martin's Press, Inc., 1991.

\end{thebibliography}

\end{document}